\documentclass[12pt]{amsart}
\usepackage[english]{babel}

\usepackage{pifont}
\usepackage{todonotes}

\usepackage{geometry}
\usepackage{amsthm,amsmath,amssymb,amsfonts,amscd,mathtools}
\usepackage[shortlabels]{enumitem}
\usepackage{hyperref,cleveref}
\usepackage{multirow,float}

\theoremstyle{plain}

\newtheorem{theorem}{Theorem}[section]
\newtheorem{theoremintro}{Theorem}

\newtheorem{lemma}[theorem]{Lemma}
\newtheorem{proposition}[theorem]{Proposition}

\newtheorem{corollary}[theorem]{Corollary}

\theoremstyle{definition}
\newtheorem{definition}[theorem]{Definition}
\newtheorem{remark}[theorem]{Remark}
\newtheorem{notation}[theorem]{Notation}

\theoremstyle{remark}
\newtheorem{example}[theorem]{Example}

\numberwithin{equation}{section}
\numberwithin{figure}{section}
\numberwithin{table}{figure}

\newcommand{\pt}[1]{\left({#1}\right)}

\newcommand{\rest}[2]{\left.{#1}\right|_{#2}}
\newcommand{\pg}[1]{\left\{{#1}\right\}}
\newcommand{\abs}[1]{\left|{#1}\right|}

\newcommand{\spos}[1]{\operatorname{SP}^{#1}}
\newcommand{\hpos}[1]{\operatorname{HP}^{#1}}
\newcommand{\wpos}[1]{\operatorname{WP}^{#1}}
\newcommand{\cone}{\mathcal{C}}
\newcommand{\sspos}[1]{\mathring{\operatorname{SP}^{#1}}}
\newcommand{\shpos}[1]{\mathring{\operatorname{HP}^{#1}}}
\newcommand{\swpos}[1]{\mathring{\operatorname{WP}^{#1}}}
\newcommand{\scone}{\mathring{\mathcal{C}}}

\newcommand{\sposh}[1]{\operatorname{SP}^{#1}_{\mathfrak h}}
\newcommand{\hposh}[1]{\operatorname{HP}^{#1}_{\mathfrak h}}
\newcommand{\wposh}[1]{\operatorname{WP}^{#1}_{\mathfrak h}}
\newcommand{\coneh}{\mathcal{C}_{\mathfrak h}}

\newcommand{\swposh}[1]{\mathring{\operatorname{WP}}^{#1}_{\mathfrak h}}
\newcommand{\sconeh}{\mathring{\mathcal{C}}_{\mathfrak h}}

\newcommand{\aldue}[3]{\,\alpha^{#1\bar{#2}\bar{#3}}\,}

\newcommand{\lambdahh}[2]{\Lambda^{{#1},{#2}}_{\mathfrak h}}
\newcommand{\lamhh}[3]{\Lambda^{{#2},{#3}}_{{\mathfrak h}_{#1}}}
\newcommand{\vol}{\operatorname{Vol}}

\newcommand{\R}{\mathbb{R}}

\newcommand{\C}{\mathbb{C}}

\newcommand{\rv}{\rho_{v_0}}

\DeclareMathOperator{\real}{Re}
\DeclareMathOperator{\imm}{Im}

\DeclareMathOperator{\rk}{rank}
\DeclareMathOperator{\tr}{tr}

\newcommand{\tic}{\ding{51}}
\newcommand{\no}{\ding{55}}

\title[Some criteria for positive forms and applications]{Some  criteria for positive forms\\and applications}

\author{Filippo Fagioli}
\address[F. Fagioli]{Dipartimento di Matematica\\Università degli Studi di Roma Tor Vergata\\Via della Ricerca Scientifica 1\\00133 Roma, Italy}
\email{fagioli@mat.uniroma2.it}

\author{Asia Mainenti}
\address[A. Mainenti]{Dipartimento di Matematica ``G. Peano''\\Universit\`{a} degli Studi di Torino \\Via Carlo Alberto 10\\10123 Torino, Italy}
\email{asia.mainenti@unito.it}

\thanks{The first author acknowledges the MUR Excellence Department Project MatMod@TOV awarded to the Department of Mathematics, University of Rome Tor Vergata, CUP E83C23000330006.
The second author is partially supported by Project PRIN 2022 \lq \lq Geometry and Holomorphic Dynamics” and by GNSAGA (INdAM).
The first author is also member of GNSAGA (INdAM)}

\begin{document}

\begin{abstract}
The aim of this paper is to gain a better understanding of weak and strong positivity for exterior forms on complex vector spaces. 
We prove a dimensionality reduction argument for positive forms, which allows us to restrict to the
case of $(2,2)$-forms in $\C^4$.
In this setting, we find criteria for weak positivity based on the associated Hermitian matrix.
As an application we prove, by duality, the strong positivity of some families of $(2,2)$-forms, already of interest in works by other authors.
\end{abstract}

\maketitle

\section{Introduction}
The theory of positive differential forms has deep roots in complex geometry, algebraic geometry and global analysis (see for instance \cite{Kod54,Lel68,Fed69}), and has been significantly influenced by the theory of Hermitian vector bundles (see \cite{Nak55,Gri69}).
In this context, the distinction between various notions of positivity for $(p,p)$-forms has played an important role, particularly in recent years.
For instance, it is now established that positively curved Hermitian metrics on holomorphic vector bundles give rise to families of positive characteristic forms.
In this setting, the notion of \textit{intermediate} positivity for forms has been studied in \cite[p.~246]{Gri69}, \cite{Li20,Fin20,Wan24}, while the concepts of \textit{strong} and \textit{weak} positivity have emerged for these characteristic forms in \cite[p.~249]{Gri69}, \cite{Gul12,DF20,Fag23} (see \cite[Sec.~3.2]{fagioli} for further details).
This highlights the motivation to look for further criteria to distinguish between different types of positivity for differential forms, which could be useful for both theoretical developments and applications.

Aside from the extremal cases of $(0,0)$ and $(n,n)$-forms, where positivity is trivially defined, $(1,1)$-forms admit a natural and well-established notion of positivity.
Namely, a real $(1,1)$-form on a complex vector space $V$ of complex dimension $n$ is called \textit{positive}, respectively \textit{strictly positive}, if its corresponding Hermitian form is positive semi-definite, respectively definite.

For intermediate values of $p$, various non-equivalent positivity concepts arise and over time the theory has led to consider essentially the three fundamental notions mentioned above: weak, intermediate, and strong.
In \cite{hk} one can find a detailed account on positivity of exterior forms on a complex vector space.
A real $(p,p)$-form $u$ is called \textit{weakly positive} if its restriction to each $p$-dimensional complex subspace $W \subset V$ gives a non-negative volume form on $W$.
If the Hermitian form on $\Lambda^{p,0} V^*$ associated to $u$ is positive semi-definite, then $u$ is called \textit{positive}. Finally, $u$ is called \textit{strongly positive} if all of its wedge products against weakly positive forms of complementary bi-degree give non-negative volume forms on $V$.
The corresponding \textit{strict} notions are defined similarly.
These conditions define cones of real forms, satisfying the following inclusions
\begin{equation*}
    \begin{cases}\begin{rcases}
        \hfil\text{weakly positive}\\
        \hfil(p,p)\text{ forms}
    \end{rcases}    \end{cases}
    \subseteq\quad
    \begin{cases}\begin{rcases}
        \hfil\text{positive}\\
        \hfil(p,p)\text{ forms}
    \end{rcases}    \end{cases}
    \subseteq\quad
    \begin{cases}\begin{rcases}
       \hfil \text{strongly positive}\\
        \hfil(p,p)\text{ forms}
    \end{rcases}    \end{cases},
\end{equation*}
and the strict ones correspond to the interior of the cones, see \cite[p.~46 and Cor.~1.10]{hk}.

If $p=1$, note that these three positivity concepts do coincide with the notion for $(1,1)$-forms mentioned above.
By duality, they coincide also if $p=n-1$.
In all other cases, weak, intermediate and strong positivity are not in general equivalent, namely the inclusions above are strict, and the study of their distinctions will be a central focus of this paper.

Intermediate {positive} forms are characterized (see \cite[Thm.~1.2]{hk}) as the forms belonging to the cone generated by \textit{elementary} forms, namely wedge products $\psi\wedge\bar\psi$, of $(p,0)$-forms $\psi$ by their conjugate.
Furthermore, this intermediate positivity condition is, in a sense, the easiest one to verify, because it can be equivalently restated in terms of positive definiteness of a Hermitian matrix.

As the name suggests, strong positivity is the most restrictive positivity notion, satisfied by powers of Hermitian metrics (\cite[Cor.~1.10]{hk}).
In this case, the cone is generated by the so called \textit{decomposable} $(p,p)$-forms, namely those forms which can be written as a wedge product $\eta_1\wedge\dots\wedge\eta_p\wedge\bar\eta_1\wedge\dots\wedge\bar\eta_p$, with $\eta_j$ a $(1,0)$-form, for all $j=1,\dots,p$.
It follows from the definition that the cone dual to that of strongly positive $(n-p,n-p)$-forms, where the duality pairing is given by the wedge product, characterizes {weakly positive $(p,p)$-forms}.
While it is not hard to construct strongly positive forms, it is quite difficult to determine whether a given real form is strongly positive, except in certain specific instances.

The situation is not better for weakly positive forms, despite this condition being quite intuitive thanks to its geometric interpretation via restrictions to suitable complex subspaces.
A further motivation to consider strictly weakly positive forms comes from the study of the so called {\it $p$-K\"ahler} manifolds.
In \cite{AA}, $p$-K\"ahler structures were introduced as closed, {\it transverse} forms, where transversality is nothing but strict weak positivity, in the notations of \cite{sullivan}.
They were further studied in \cite{AB,ABcurrents} and related works by the same authors.
One of the open problems related to this subject is the lack of examples, especially for the case of $2$-K\"ahler structures on non-K\"ahler manifolds of dimension $4$, which is the first {\it non-trivial} case.
Some obstructions to the existence of such structures can be found in \cite{AlessandriniLieGroups,fm}, so we hope that criteria for strict weak positivity will help to construct explicit examples.
This, together with further explicit constructions of $(n-2)$-K\"ahler structures, in dimension $n\ge4$, would also be useful to pursue the work initiated in \cite{fmarxiv} to construct non-K\"ahler examples of {\it Hodge-Riemann balanced structures}, see \cite{cw}.

In light of this, the main aim of this paper is to better understand the three positivity notions we discussed. 
To do so, we started with a study of criteria relating positivity on $V$ to positivity on any subspace $\mathfrak h$ of codimension $1$.
Once one fixes a complement of $\mathfrak h$ in $V$, the space of $(p,q)$-forms on $\mathfrak h$ can be identified with a subspace of $\Lambda^{p,q}V
^*$, which we denote with $\lambdahh pq$.
Then, fix any $\alpha \in \Lambda^{1,0} V^* \setminus \lambdahh10$ and for each $(p,p)$-form $\Omega$ on $V$, we can write
\begin{equation*}
    \Omega=\Xi+\alpha\wedge\eta+\bar\alpha\wedge\bar\eta+i\alpha\wedge\bar\alpha\wedge\vartheta
\end{equation*}
for some $\Xi\in\lambdahh pp$, $\eta\in\lambdahh {p-1}p$, $\vartheta\in\lambdahh {p-1}{p-1}$.
The first main result of this paper is the following, see \Cref{nonso,cono+conoincono}.
\begin{theoremintro}\label{introPos}
    If $\Omega$ has a positivity type (weak, intermediate or strong), then $\Xi$ and $\vartheta$ have the same positivity type.
    Also, if $\Omega$ is in the interior of the relative positive cone, then $\Xi$ and $\vartheta$ are in the respective interior cones.

    Moreover, if $\eta=0$, the converse of both the statements holds.
\end{theoremintro}

This dimensional reduction is telling us that information on the low dimension cases could be enough to recover information in higher dimension.
In complex dimension $3$, the cones always coincide, so we have a full understanding of the problem.
Motivated by this, we start by the first {\it non-trivial} case, namely the case of $(2,2)$-forms in complex dimension~$4$, where the problem of determining whether a given form is weakly or strongly positive remains challenging.

In general, unlike strong positivity, where a criterion for testing it is currently missing, weak positivity has an algebraic characterization in terms of Pl\"ucker equations.
Indeed, it relies on the identification of decomposable $(2,0)$-forms as the image, via the Pl\"ucker embedding, of the Grassmannian of $2$-planes in $V \cong \mathbb{C}^4$ into the projectivization of the space of $(2,0)$-forms, cf. \cite{BlockiPlis}.
The application of this characterization to the context of positive forms was originally pointed out by Dinew in an unpublished work, together with insights on how to generalize it to higher dimension, and was later revisited in \cite[Thm.~3]{BlockiPlis}.
Despite this condition being rarely easy to check, it is still extremely useful and it has been the key ingredient to prove the following, see \Cref{3antidiag} for the precise and detailed statement.
\begin{theoremintro}\label{thmBintro}
    If $(a_{jk})$ represents the $6 \times 6$ matrix associated to a given weakly positive $(2,2)$-form on $V \cong \mathbb{C}^4$ w.r.t. the standard basis induced by any basis of $V^*$, then the elements in the anti-diagonal, namely $a_{jk}$ such that $j+k=7$, satisfy the relation
\begin{equation}\label{gigettino}
\abs{a_{j,k}}\le\frac{a_{j,j}+a_{k,k}+a_{l,l}+a_{m,m}}2
\end{equation}
for $(l,m) \neq (j,k)$ with $l + m = 7$.
{Moreover, this upper bound is sharp.}
\end{theoremintro}
The standard basis in the claim is defined by \eqref{base20}.

\smallskip
As an application, we achieved to carry on the study initiated in \cite{BlockiPlis}, where the authors exhibit a family of $(2,2)$-forms on $\C^4$, depending on a complex parameter, and determine for which values of the parameter the given form is weakly positive.
This family is of particular interest, as it contains two weakly positive forms whose wedge product is not positive.
We extend \cite[Prop.~4]{BlockiPlis} in \Cref{alphaa}, with a characterization of strict weak positivity, strong positivity and strict strong positivity for that family of forms, as the parameter varies.

We want to emphasize that in the last part of the paper we used the inequality~\eqref{gigettino} concerning weakly positive forms to establish, {via} duality of cones, the strong positivity of certain explicit $(2,2)$-forms.
This approach turned out to be significantly simpler than verifying strong positivity by searching for a \textit{diagonal} expression in terms of decomposable forms.
In light of this, we believe that \Cref{thmBintro} could help, by duality, with the understanding of strong positivity.
 Moreover, together with \Cref{introPos}, this could serve as a first step towards the development of positivity criteria in complex dimension~$5$ and higher.

\bigskip
The organization of the paper is as follows. 

\medskip
After reviewing, in \Cref{secPos}, the basic notions on positivity for $(p,p)$-forms, in \Cref{chPos} we focus on the relation of positivity on a complex vector space $V$ with positivity on complex hyperplanes.
More precisely, the cones of positive forms on the subspace are immersed in the respective cones of positive forms on $V$, see \Cref{coneshp,coneswedgeh}.
This, together with an inverse construction, provides the proof of \Cref{introPos}.

In \Cref{wp224} we discuss the case of $(2,2)$-forms in complex dimension $4$.
We provide necessary conditions for weak positivity, see \Cref{+diag,3x3subm} and \Cref{3antidiag}.
To conclude, we make use of such results to discuss strong positivity on a family of forms, including that of \cite[Prop.~4]{BlockiPlis}, see \Cref{wposb} and \Cref{alphaa}.

\section{Preliminaries}\label{secPos}

We will now discuss some positivity notions on exterior forms on a complex vector space.
The primary reference we consider is \cite{hk}, from which we will adopt the majority of the  notations.
In literature, several authors considered the same positivity conditions, using different notations and terminology, see for instance \cite[Sec.~III.1]{demailly} and \cite[Rmk.~1.3]{fagioli}.

Let $V$ be a complex vector space of dimension $n$ and denote with $\Lambda^{p,q}\coloneqq\Lambda^{p,q}V^*$ the space of $(p,q)$-forms over $V$.
From here on, when talking about vector spaces we will always intend them to be complex vector spaces, unless differently specified.
Similarly, the dimension of a complex vector space will be the complex dimension.

\begin{definition}\label{elemForms}
A $(k,0)$-form $\psi$  on $V$  is called \textit{decomposable} if it can be written as $\psi=\mu_1\wedge\dots\wedge\mu_k$, with $\mu_1,\dots,\mu_k\in\Lambda^{1,0}$.
A $(k,k)$-form $\Phi$  on $V$ is said to be \textit{decomposable} if it can be written as $\Phi=\psi\wedge\bar\psi$, with $\psi$ a decomposable $(k,0)$-form.
\end{definition}
All $(k,0)$-forms, for $k=1,n$, are trivially decomposable, and actually the same holds for $k=n-1$, by duality.

\begin{remark}\label{decomPlucker}
Let $\mathbb{G}_k(V)$ denote the Grassmannian of $k$-dimensional subspaces of $V$.
It is well-known, see for instance \cite{harris}, that the set of decomposable $k$-forms on $V$ can be identified, modulo $\C^*$, with the image of $\mathbb{G}_k(V)$ in $\mathbb{P}(\Lambda^k V)$ via the Pl\"ucker embedding.
This gives us a characterization of decomposable forms in terms of their coefficients in a suitable basis, for which they have to satisfy the so called {Pl\"ucker equations}.
Clearly, the same holds for $(k,0)$-forms.
\end{remark}

Recall that a $(p,q)$-form is called \textit{real} if it equals its complex conjugate.
In this case, it clearly follows that $p=q$.

In order to introduce positivity on exterior forms, we need to fix a \textit{volume form} on $V$, that is a  real, non-zero, $(n,n)$-form on $V$, which we will denote by $\vol$.
For instance, a natural way to choose a volume form is to fix a basis $\pg{\alpha^j}_{j=1}^n$ of $(1,0)$-forms on $V$ and define 
\begin{equation*}
    \vol\coloneqq i\,\alpha^1\wedge\alpha^{\bar1}\wedge\cdots\wedge i\,\alpha^n\wedge\alpha^{\bar{n}}=i^{n^2}\alpha^{1\cdots n\bar1\cdots\bar n},
\end{equation*}
where we are using the notation $\alpha^{j_1\cdots j_k}\coloneqq\alpha^{j_1}\wedge\dots\wedge\alpha^{j_k}.$
Thus, every $(n,n)$-form on $V$ is a complex multiple of $\vol$, and real $(n,n)$-forms are multiples of $\vol$ by a real multiplicative constant.
This gives rise to the notion of positivity for  $(n,n)$-forms.
\begin{definition}\label{posnn}
A real $(n,n)$-form $\nu$ is \textit{positive} if it is a {non-negative} multiple of $\vol$, and $\nu$ is \textit{strictly positive} if the multiplicative constant is positive.
\end{definition}

In the setting of complex manifolds of dimension $n$, where by $V$ one considers  the tangent space to the manifold at every fixed point, a standard choice of a volume form is the top power of the fundamental $(1,1)$-form associated to some Hermitian metric.
Under this choice it is easy to see that the $n$th power of any other Hermitian metric is strictly positive.

\smallskip
With this in mind, we will now see a  first positivity condition on $(p,p)$-forms, with $1\le p\le n$, which  is one way to generalize \Cref{posnn}.

\begin{definition}\label{SPos}
    A $(p,p)$-form $\Omega$ is \textit{strongly positive} if it can be written as a sum of decomposable $(p,p)$-forms,  namely
    \begin{equation*}
        \Omega=i^{p^2}\sum_j\psi_j\wedge\bar{\psi}_j, \quad\psi_j\in\Lambda^{p,0}\mbox{ decomposable.}
    \end{equation*}
Strongly positive forms of bidegree $(p,p)$ define a cone, that we will denote with $\spos p$.
A form belonging to the interior of this cone will be called \textit{strictly strongly positive}.
\end{definition}
Once again, for the case of complex manifolds, this is particularly significative in relation to Hermitian metrics, as $p$th powers of Hermitian metrics on complex manifolds are strictly strongly positive, for $0\le p\le n$.

Unlike the case of $(n,n)$-forms, for $p<n$, we can define a few different conditions of positivity for real forms of bidegree $(p,p)$.
However, the extremal values $p=1,n-1$ behave similarly to $p=n$, as all of them are equivalent to \Cref{SPos}.

\begin{definition}\label{defPos}
Let $\Omega$ be a real $(p,p)$-form on $V$.
$\Omega$ is said to be \textit{positive}, or \textit{Hermitian positive}, if the $(n,n)$-form 
\begin{equation*}
     i^{\pt{n-p}^2}\,\Omega\wedge\eta\wedge\bar\eta
\end{equation*}
is positive,  for all $\eta\in\Lambda^{n-p,0}$.
If for all such $\eta$, $ i^{\pt{n-p}^2}\,\Omega\wedge\eta\wedge\bar\eta$ is a volume form, $\Omega$ will be called \textit{strictly positive}, or \textit{strictly Hermitian positive}.

If for every non-zero, decomposable $(n-p,0)$-form $\psi$
\begin{equation*}
     i^{\pt{n-p}^2}\,\Omega\wedge\psi\wedge\bar\psi\in\Lambda^{n,n}
\end{equation*}
is positive, $\Omega$ is said \textit{weakly positive}.
If moreover $ i^{\pt{n-p}^2}\,\Omega\wedge\psi\wedge\bar\psi$ is strictly positive for all such $\psi$, $\Omega$ is \textit{strictly weakly positive}. %, or \textit{transverse}. 
\end{definition}
To avoid confusion, we refer to positive and strictly positive forms as to Hermitian positive and strictly  Hermitian positive, following the notation of \cite{fagioli}.
We will see in what follows that this terminology is justified by the fact that this type of positivity is equivalent to the positive definiteness of some Hermitian form.

\begin{remark}\label{wposSsp}
    The last condition in \Cref{defPos} was studied in \cite{hk} as strict weak positivity, and independently by Sullivan, in \cite{sullivan}, as transversality with respect to the cone of the complex structure. 
    
    Strictly weakly positive forms are characterized by the condition of being strictly positive volume forms, when restricted to vector spaces of dimension $p$.
    Similarly, a $(p,p)$-form is weakly positive if and only if its restriction to every such subspace is positive, in the sense of \Cref{posnn}.
\end{remark}

Just like for strong positivity, Hermitian positive forms, respectively weakly positive forms, define a cone, denoted $\hpos p$, respectively $\wpos p$.
Moreover, it follows from the definition that the cone of weakly positive forms $\wpos p$ is the dual cone to that of strongly positive forms $\spos{n-p} $, while the cone of Hermitian positive forms is self-dual, in the sense that the dual cone to $\hpos p$ is $\hpos{n-p}$.
In particular, as one can see from this last condition, the cone structure of Hermitian positive $(p,p)$-forms is generated by the set of elementary forms $\pg{\eta\wedge\bar\eta,\,\eta\in\Lambda^{p,0}}$.
As for the strict conditions, they correspond to the interior of each respective cone.
It is straightforward that the following inclusions hold
\begin{equation*}
    \wpos p\subset\hpos p\subset\spos p,
\end{equation*}
and that the inclusions are strict, for $1<p<n-1$, whereas for the extremal values $p=1,n-1$, all the inclusions turn out to be equalities.
Of course, the situation is the same for the interior cones.

Given $\Omega\in\Lambda^{p,p}$ real, one can consider the Hermitian form $Q_\Omega$ on $\Lambda^{n-p,0}$ defined by
\begin{equation*}
     i^{\pt{n-p}^2}\,{\Omega\wedge\beta\wedge\bar\gamma}=Q_\Omega(\beta,\gamma)\text{Vol},
    \quad\beta,\gamma\in\Lambda^{n-p,0}
\end{equation*}
Note that $Q_\Omega(\beta,\gamma)=\star\, i^{\pt{n-p}^2}\,\Omega\wedge\beta\wedge\bar\gamma$, where $\star$ is the Hodge star operator with respect to any Hermitian metric with associated volume form $\text{Vol}$.

\begin{remark}\label{posQ}
    Some of the positivity conditions above can be expressed in terms of  positive definiteness of the associated Hermitian form $Q_{\Omega}$, when restricted to suitable subspaces of  $\Lambda^{n-p,0}$:
    \begin{enumerate}[label=\roman*.]
    \item
    $\Omega$ is {weakly positive} if and only if the restriction of $Q_{\Omega}$ to decomposable forms is positive semi-definite.
    \item
    $\Omega$ is {strictly weakly positive} if and only if the restriction of $Q_{\Omega}$ to decomposable forms is positive definite.
    \item $\Omega$ is {Hermitian positive} if and only if $Q_{\Omega}$ is positive semi-definite on $\Lambda^{n-p,0}$.
    \item\label{shposEigen} $\Omega$ is {strictly Hermitian positive} if and only if $Q_{\Omega}$ is positive definite on $\Lambda^{n-p,0}$.
\end{enumerate}
\end{remark}

Hermitian and strict Hermitian positivity can also be restated in terms of the so called {\it eigenvalues} of a real form.
\begin{proposition}[{\cite[Thm.~1.2]{hk}}]\label{hkeigenvalues}
    Any real $(p,p)$-form $\Omega$ can be written as 
\begin{equation*}%\label{eigenvalues}
    \Omega= i^{\pt{n-p}^2}\sum_j\lambda_j\,\psi^j\wedge\psi^{\bar j},
\end{equation*}
for a suitable normalized orthogonal basis $\pg{\psi^1,\dots,\psi^N}$ of $\Lambda^{p,0}$ and $\lambda_j\in\R$ are the eigenvalues of $\Omega$.
For a fixed normalization of the orthogonal basis, 
% for instance $\abs{\psi^j}^2=2^p$,
the $\lambda_j$ and their multiplicity are unique, and so are the eigenspaces $V_\lambda\coloneqq\text{span}\pg{\psi^j\colon\lambda_j=\lambda}$.
Moreover, the number of eigenvalues that are positive, negative or null is independent of the choice of inner product.
\end{proposition}

The eigenvalues of $\Omega$ are exactly the eigenvalues of $Q_\Omega$, and the $Q_\Omega$-eigenspace $W_\lambda$ corresponding to the eigenvalue $\lambda$ is the dual of $V_\lambda$, via the Hodge star operator.
As a consequence, we find the following relations between the positivity of $\Omega$ and the positivity of its eigenvalues.

\begin{remark}\label{rmkboh}
    A real $(p,p)$-form $\Omega$ is Hermitian positive if and only if $\lambda_j\geq0$, for all $j$, and it is strictly Hermitian positive if and only if $\lambda_j>0$, for all $j$.
\end{remark}
Note that the weaker positivity conditions do not give restrictions on the positivity of the eigenvalues.
Weakly positive forms can in fact admit both positive and negative eigenvalues, as the following example shows.

\begin{example}[\cite{BlockiPlis}]
Let $\pg{\alpha^1,\dots,\alpha^4}$ be a basis of $\Lambda^{1,0}\C^4$, and consider the real $(2,2)$-form 
\begin{equation*}
    \Omega_a=\sum_{1\le j<k\le4}\alpha^{jk\bar j\bar k}+a\,\pt{\alpha^{12\bar3\bar4}+\alpha^{34\bar1\bar2}},
\end{equation*}
with $a\in\R$.
By \cite[Prop.~4]{BlockiPlis}, we know that $\Omega$ is weakly positive for $\abs{a}\le2$.
A simple computation shows that one can write
\begin{equation*}
\begin{aligned}
    \Omega_a=&\alpha^{13}\wedge\aldue{}13+\alpha^{14}\wedge\aldue{}14+\alpha^{23}\wedge\aldue{}23+\alpha^{24}\wedge\aldue{}24\\
   & +\frac{1+a}{2}\pt{\alpha^{12}+\alpha^{34}}\wedge\pt{\aldue{}12+\aldue{}34}
    +\frac{1-a}{2}\pt{\alpha^{12}-\alpha^{34}}\wedge\pt{\aldue{}12-\aldue{}34},
\end{aligned}\end{equation*}
so for $1<|a|\le2$, $\Omega_a$ is a weakly positive form, with one negative eigenvalue.
\end{example}

The number of non-null eigenvalues of a real $(p,p)$-form is called \textit{rank}, and it is at most $N=\dim\Lambda^{p,0}$.
One can characterize the condition of being in the interior of the cone $\hpos p$ in terms of the rank, in the sense that strictly Hermitian positive forms are Hermitian positive forms of maximal rank, as follows by \cref{shposEigen} in \Cref{posQ} and \Cref{rmkboh}. 

\section{Dimensionality reduction}\label{chPos}

We now compare positivity notions on a vector space to the respective notions on a subspace of   codimension $1$.
Let $V$ be a vector space of dimension $n$, and fix a subspace $\mathfrak h$ of  dimension $n-1$.

\subsection{Positivity conditions, from \texorpdfstring{$\mathfrak h$}{h} to \texorpdfstring{$V$}{V}}\label{htoV}

Consider the cones of positive forms on $V$, $\spos p,\hpos p,\wpos p\subset\Lambda^{p,p}$.
We can interpret the analogous cones $\spos p_\mathfrak h,\hpos p_\mathfrak h,\wpos p_\mathfrak h$ on $\mathfrak h$ as follows.
The cone $\spos{p}_{\mathfrak h}$ of strongly positive $(p,p)$-forms on $\mathfrak h$ is defined as the cone generated by decomposable $(p,p)$-forms on $\mathfrak h$. %, in the sense of \Cref{elemForms}.
Similarly, we define the cone of Hermitian positive $(p,p)$-forms on $\mathfrak h$, denoted by $\hpos p_{\mathfrak h}$, as the cone with generating set 
\begin{equation*}
   \mathcal{P}_\mathfrak h= \pg{i^{p^2}\varphi\wedge\bar\varphi,\,\varphi\in\lambdahh p0}.
\end{equation*}
% namely elementary $(p,p)$-forms.
Lastly, weakly positive forms on $\mathfrak h$ are defined by duality with respect to strongly positive forms, in the following sense.
For a fixed volume form $\vol _\mathfrak h$ on $\mathfrak h$, we define
\begin{equation*}
    \wposh p\coloneqq\pg{\Omega\in\lambdahh pp, \, \Omega\wedge\psi=c\vol_\mathfrak h,\,c\ge0,\,\forall\,\psi\in\sposh{(n-1)-p}},
\end{equation*}
where $\dim_\C\mathfrak h=n-1$.
Similarly to \Cref{secPos}, the strict conditions are equivalent to be in the interior of the cones.
For instance,
\begin{equation*}
    \swposh p\coloneqq\pg{\Omega\in\lambdahh pp, \, \Omega\wedge\psi=c\vol_\mathfrak h,\,c>0,\,\forall \,0\neq\psi\in\sposh{(n-1)-p}}.
\end{equation*}
Once again, strictly strongly positive and strictly Hermitian positive forms have maximal rank, as real $(p,p)$-forms on $\mathfrak h$.
Note that the maximal possible rank for a real $(p,p)$-form on $\mathfrak h$ equals the dimension of $\lambdahh p0$, namely $\binom{n-1}{p}$.
We recall that weak positivity has a geometric interpretation in terms of complex subspaces, namely a $(p,p)$-form on $\mathfrak h$ is weakly positive, respectively strictly weakly positive, if and only if its restriction to any complex subspace of $\mathfrak h$ of  dimension $p$ is a non-negative, respectively positive, volume forms on said  subspace.

\smallskip
To relate positivity on $V$ and positivity on $\mathfrak h$, we have to fix a complement of $\mathfrak h$ in $V$, namely a preferred $v_0\in V\setminus\mathfrak h$. %, or equivalently $\alpha\in V^*\setminus\mathfrak h^*$.
Thus, for all $p,q=1,\dots,n-1$, $v_0$ induces an immersion $\rho_{v_0}$ of the space of $(p,q)$-forms on $\mathfrak h$ in the space of $(p,q)$-forms on $V$, defining, for all $\eta\in\lambdahh pq$, $\rho_{v_0}\eta\in\Lambda^{p,q}$ to be the extension of $\eta$ vanishing in $v_0$ and $\bar v_0$.
In other words, for $v_1,\dots,v_{p+q}\in V$, we can write $v_j=c_jv_0+v_j^{\mathfrak h}$, for unique $c_j\in\C$ and $v_j^{\mathfrak h}\in\mathfrak h$, and
\begin{equation}\label{lahk}
     \rho_{v_0}\eta\pt{v_1,\dots,v_p,\bar v_{p+1},\dots,\bar v_{p+q}}\coloneqq\eta\pt{v_1^{\mathfrak h},\dots,v_p^{\mathfrak h},\bar v_{p+1}^{\mathfrak h},\dots,\bar v_{p+q}^{\mathfrak h}}.
\end{equation}
Conversely, given $\mu\in\Lambda^{p,q}$, we will say that $\eta$ is in $\rho_{v_0}\pt{\lambdahh pq}$ if and only if the contraction $\iota_{v_0}\iota_{\bar v_0}\mu=0$.
Note also that the inverse of $\rv\colon\lambdahh pq\to\rv\Lambda^{p,q}$ is the restriction to $\mathfrak h$, denoted with $\rest{\cdot}{\mathfrak h}$.

The upshot is that the map $\rho_{v_0}$ allows to identify $\lambdahh pq$ with the subspace  $\rho_{v_0}\pt{\lambdahh pq}$ of $\Lambda^{p,q}$.
% , and similarly one can extend $\rho_{v_0}$ to the space $\lambdahh pq$ of $(p,q)$-forms on $\mathfrak h$, and get the further relation 
% %\begin{equation*}
%      \rho_{v_0}\pt{\lambdahh pq}=\rho_{v_0}\pt{\lambdah{p+q}}\cap\Lambda^{p,q}.
% %\end{equation*}
 In particular, the cones of positive $(p,p)$-forms on $\mathfrak h$ are immersed in $\Lambda^{p,p}$, so one can investigate the relation between such cones and those of positive forms on $V$.

\begin{lemma}\label{coneshp}
    For all $p=1,\dots,n-1$, and $\cone=\spos{},\hpos{},\wpos{}$, and for every $v_0\in V\setminus\mathfrak h$,
    \begin{equation*}
        \rho_{v_0}\pt{\coneh^p}\subset\cone^p\setminus\mathring{\cone}^p.
    \end{equation*}
\end{lemma}

 \begin{proof}
 Firstly, we can see that the inclusion holds for the cone $\hposh p$ of Hermitian positive forms on $\mathfrak h$, for all $p=1,\dots,n-1$, because we can write explicitly 
 \begin{equation*}
     \hposh p=\pg{i^{p^2}\sum_{j}\varphi_j\wedge\bar\varphi_j,\,\varphi_j\in\lambdahh p0},
 \end{equation*}
 and see that every Hermitian positive $(p,p)$-form on $\mathfrak h$ is mapped via $\rv$ to a  Hermitian positive $(p,p)$-form on $V$, by linearity of $\rho_{v_0}$.
 We can see similarly that $\rv\sposh p\subset\spos p$, for all $p=1,\dots,n-1$, once again by linearity of $\rv$, because
 \begin{equation*}
     \sposh p=\pg{i^{p^2}\sum_{j}\varphi_j\wedge\bar\varphi_j,\,\varphi_j\in\lambdahh p0, \text{ decomposable}}.
 \end{equation*}
The fact that
\begin{equation*}
     \rv\pt{\spos p_\mathfrak h}\subset\spos p\setminus\mathring{\spos p},\quad
     \rv\pt{\hpos p_\mathfrak h}\subset\hpos p\setminus\shpos p,
\end{equation*}
follows by a simple argument on the rank. 
Indeed, we recall that for these two cones, being in the interior implies being of maximal rank, but the maximal rank for $(p,p)$-forms on $\mathfrak h$ is strictly less than the maximal rank for $(p,p)$-forms on $V$.

The case $\cone=\wpos{}$ is a direct consequence of \eqref{lahk}, using the characterization in \Cref{wposSsp}.
It is clear that for $\eta\in\wposh p$, $\rv\eta$ is non-negative when restricted to complex vector spaces of $V$ of complex dimension $p$.
As for the strict condition, consider any $W$, complex subspace of $\mathfrak h$ of dimension $p-1$, and $W_0\coloneqq W\oplus v_0$.
A simple computation shows that $\rest{\rv\eta}{W_0}=0$, for all $\eta\in\lambdahh pp$.
% The case $\cone=\wpos{}$ can be proved via the interpretation of weakly positive forms in terms of their restrictions to  subspaces of suitable dimension.
% Indeed, elements of $\wposh p$ are $(p,p)$-forms on $\mathfrak h$ such that their restriction to complex subspaces of $\mathfrak h$ of  dimension $p$ are non-negative volume forms on said subspace.
% The key point here is to distinguish  subspaces of $V$ of  dimension $p$, depending on their position with respect to $\mathfrak h$.
% Indeed, every complex $p$-subspace of $V$ is either fully contained  in $\mathfrak h$, or its  intersection with $\mathfrak h$ is a  subspace of dimension $p-1$.
% With this distinction, it is easy to see that a weakly positive form on $\mathfrak h$ is also a weakly positive form on $V$, as its restriction to subspaces contained in $\mathfrak h$ is a non-negative volume form, by assumption, whereas its restriction to any other subspace is $0$, by \eqref{lahk}, because such subspaces have directions in $V\setminus\mathfrak h$.
% The last part of this argument also proves that no weakly positive form on $\mathfrak h$ can be a strictly positive form on $V$, since we showed that there are always complex subspaces of $V$ on which such a form vanishes.
 \end{proof}

We just saw how, for every fixed $v_0\in V\setminus\mathfrak h$, we can use the immersion map $\rv$ to see positive $(p,p)$-forms on $\mathfrak h$ as positive $(p,p)$-forms on $V$.
We now want to use $\rv$ to construct positive $(q+1,q+1)$-forms on $V$, starting from positive $(q,q)$-forms on $\mathfrak h$.

% Note that the construction above allows to define the \textit{dual} of $v_0$ as the only $\alpha_0\in\Lambda^1\setminus\rv\pt{\lambdah1}$ such that $\alpha_0(v_0)=1$.

% if we choose $v_0$ to be a $(1,0)$-vector, its dual $\alpha\coloneqq v\check{}$ is a form in $\Lambda^{1,0}\setminus\lambdahh 10$, and by dimensional reasons
% \begin{equation*}
% \Lambda^{1,0}=\C\alpha\oplus\lambdahh 10
% \end{equation*}
% This, together with \Cref{coneshp}, leads us to a construction of $(q+1,q+1)$-forms on $V$, starting from positive $(q,q)$-forms on $\mathfrak h$.

\begin{lemma}\label{coneswedgeh}
    For every choice of $\alpha\in\Lambda^{1,0}\setminus\rv\pt{\lambdahh 10}$, for $\cone=\spos{},\hpos{},\wpos{}$, the following inclusion holds
    \begin{equation*}
         i\,\alpha\wedge\bar\alpha\wedge\rv\pt{\mathcal{C}^{q}_\mathfrak h}\subset
    \mathcal{C}^{q+1}\setminus\mathring{\cone}^{q+1},
    \end{equation*}
    for $q=1,\dots,n-2$, and 
    \begin{equation*}
        i\,\alpha\wedge\bar\alpha\wedge\rv\pt{\coneh^{n-1}}=
    \cone^{n},\quad
    i\,\alpha\wedge\bar\alpha\wedge\rv\pt{\sconeh^{n-1}}=
    \scone^{n}.
    \end{equation*}
\end{lemma}
For a part of the proof, we will make use of the definition of the interior cones via the dual cones, and we will need the following remark.

\begin{remark}\label{rmkDualCone}
    For $k=1,\dots,n$, consider the dual cones
    \begin{equation*}
        \pt{\mathcal{C}^{k}}\check{}=\begin{cases}
            \wpos{n-k},  &   \mathcal{C}^k=\spos k, \\
            \hpos{n-k},  &   \mathcal{C}^k=\hpos k, \\
            \spos{n-k},  &   \mathcal{C}^k=\wpos k,
        \end{cases}
        \quad\quad
        \pt{\mathcal{C}^{k}_\mathfrak h}\check{}=\begin{cases}
            \wpos{n-1-k}_\mathfrak h,  &   \mathcal{C}^k_\mathfrak h=\spos k_\mathfrak h, \\
            \hpos{n-1-k}_\mathfrak h,  &   \mathcal{C}^k_\mathfrak h=\hpos k_\mathfrak h, \\
            \spos{n-1-k}_\mathfrak h,  &   \mathcal{C}^k_\mathfrak h=\wpos k_\mathfrak h.
        \end{cases}
    \end{equation*}
   By \Cref{coneshp}, we have $\rv\pt{\mathcal{C}^{k-1}_{\mathfrak h}}\check{}\subset\pt{\mathcal{C}^k}\check{}$.
\end{remark}

We can now prove \Cref{coneswedgeh}.

\begin{proof}[Proof of \Cref{coneswedgeh}]
From \Cref{coneshp}, we get
\begin{equation}\label{eqboh}
    i\,\alpha\wedge\bar\alpha\wedge\rv\pt{\mathcal{C}^{q}_\mathfrak h}\subset
    i\,\alpha\wedge\bar\alpha\wedge\mathcal{C}^{q}\subset
    \mathcal{C}^{q+1},
\end{equation}
where the last inclusion holds because $i\,\alpha\wedge\bar\alpha$ is a strongly positive $(1,1)$-form.
To prove that no form in $i\,\alpha\wedge\bar\alpha\wedge\rv\pt{\mathcal{C}^{q}_\mathfrak h}$ can be in the interior of $\cone^{q+1}$, we use \Cref{rmkDualCone}.
 From \eqref{eqboh} we get that, for all $k=1,\dots,n$,
\begin{equation}\label{cone(k)h}
    i\,\alpha\wedge\bar\alpha\wedge\rv\pt{\mathcal{C}^{k}_{\mathfrak h}}\check{}\subset\pt{\mathcal{C}^k}\check{}.
    \end{equation}
Now, for $\psi\in\rv\pt{\coneh^q}$, by duality, $i\,\alpha\wedge\bar\alpha\wedge\psi$ is in $\mathring{\cone}^{q+1}$  if and only if, for all $\phi\in\pt{\mathcal{C}^{q+1}}\check{}$, 
\begin{equation*}
    i\,\alpha\wedge\bar\alpha\wedge\psi\wedge\phi=c\,\vol,
\end{equation*}
for some $c>0$.
However, we can choose $\gamma\in\rv\pt{\mathcal{C}^{q+1}_{\mathfrak h}}\check{}$, so that  $\phi_0=i\,\alpha\wedge\bar\alpha\wedge\gamma\in\pt{\mathcal{C}^{q+1}}\check{}$ by \eqref{cone(k)h} and
\begin{equation*}
    i\,\alpha\wedge\bar\alpha\wedge\psi\wedge\phi_0=
    i\,\alpha\wedge\bar\alpha\wedge\psi\wedge i\,\alpha\wedge\bar\alpha\wedge\gamma=0,
\end{equation*}
proving that $ i\,\alpha\wedge\bar\alpha\wedge\psi$ cannot be in $\mathring{\cone}^{q+1}$.

The last part of the statement is true as \Cref{p0formsh}, for $p=n$, implies in particular that if $\vol_\mathfrak h$ is a volume form on $\mathfrak h$, then 
\begin{equation*}
% \begin{aligned}
    \text{Vol}_V=\frac i2\alpha\wedge\bar\alpha\wedge\rv\pt{\text{Vol}_\mathfrak h}
% \end{aligned}
\end{equation*}
is a volume form on $V$.
\end{proof}

We showed two ways to find positive forms on $V$ starting from positive forms on $\mathfrak h$.
However, both of these constructions give forms that cannot be positive in the strict sense.
In the next subsection we will see how on can obtain a relation for the interior cones, combining \Cref{coneshp,coneswedgeh}, together with a construction that goes in the converse sense, from forms on $V$ to suitable restrictions and contractions, interpreted as forms on $\mathfrak h$.

\subsection{Positivity conditions, from \texorpdfstring{$V$}{V} to \texorpdfstring{$\mathfrak h$}{h}}\label{Vtoh}

With the same notation as above, we fix $v_0\in V\setminus \mathfrak h$, $\alpha\in\Lambda^{1,0}\setminus\rv\lambdahh10$ such that $\alpha(v_0)=1$, and let $\Omega$ be a real $(p,p)$-form on $V$.
Define
\begin{equation}\begin{aligned}
    &\Xi\coloneqq\rv\pt{\rest{\Omega}{\mathfrak h}}\in\rv\lambdahh pp,\\
    &\eta\coloneqq\rest{\pt{\iota_{v_0}\Omega}}{\mathfrak h}\in\lambdahh{p-1}p,\\
    &\vartheta\coloneqq\iota_{v_0}\iota_{\bar v_0}\Omega\in\rv \lambdahh{p-1}{p-1},
\end{aligned}
\end{equation}
with $\Xi,\Psi$ real.
Then,
\begin{equation}\label{restOmega}
    \Omega=\Xi+\alpha\wedge\rv\eta+\bar\alpha\wedge\rv\bar\eta+i\alpha\wedge\bar\alpha\wedge\vartheta.
\end{equation}
It is clear that $\rest{\Omega}{\mathfrak h}$ has the same positivity as $\Omega$, and using the results from the previous section, we are able to prove that the same holds for $\rest{\vartheta}{\mathfrak h}$.

\begin{theorem}\label{nonso}
    If $\Omega$ is in $\mathcal C^p$, with $\mathcal{C}^p=\spos p,\hpos p,\wpos p$, then $\rest{\Omega}{\mathfrak h}\in\mathcal{C}_{\mathfrak h}^p$ and $\rest{\vartheta}{\mathfrak h}\in\mathcal{C}_{\mathfrak h}^{p-1}$.
    Similarly, if $\Omega$ is in the interior of $\mathcal{C}^p$, then $\rest{\Omega}{\mathfrak h}$ and $\rest{\vartheta}{\mathfrak h}$ are in the respective interior cones.
\end{theorem}

\begin{proof}%[Proof of \Cref{nonso}]
    Assume $\Omega$ is in $\mathcal{C}^p$ and let us begin proving that $\rest{\vartheta}{\mathfrak h}\in\mathcal{C}^{p-1}_{\mathfrak h}$.
    Of course $\rest{\vartheta}{\mathfrak h}$ is a real form, because $\Omega$ is, so for all $\Psi\in\pt{\mathcal{C}^{p-1}_{\mathfrak h}}\check{}\subset\Lambda^{2\pt{n-p}} _{\mathfrak h}$, we have $\rest{\vartheta}{\mathfrak h}\wedge\Psi=c\vol_{\mathfrak h}$, for some $c\in\R$.  
    We need to show that $c$ is always non-negative.
    Consider 
    \begin{equation}\label{sette}
    i\alpha\wedge\bar\alpha\wedge\rv\pt{\rest{\vartheta}{\mathfrak h}\wedge\Psi}=i\alpha\wedge\bar\alpha\wedge\vartheta\wedge\rv\Psi=\Omega\wedge\rv\Psi,
    \end{equation}
    where the last equality holds by dimensional reasons, as $\rv\eta\wedge\rv\Psi\in\rv\Lambda^{n-1,n}_\mathfrak h=0$ and $\Xi\wedge\rv\Psi\in\rv\Lambda^{n,n}_\mathfrak h=0$.
    By \Cref{rmkDualCone}, $\rv\Psi\in\pt{\cone^p}\check{}$, so  $\Omega\wedge\rv\Psi$ is a non-negative volume form, namely it is a positive multiple of  $\text{Vol}_V=\frac i2\alpha\wedge\bar\alpha\wedge\rv\text{Vol}_\mathfrak h$.
    From \eqref{sette} we then get that
    \begin{equation*}
        i\alpha\wedge\bar\alpha\wedge\rv\pt{c\vol_{\mathfrak h}}=i\alpha\wedge\bar\alpha\wedge\rv\pt{\rest{\vartheta}{\mathfrak h}\wedge\Psi}=\Omega\wedge\rv\Psi
    \end{equation*}
    is a positive multiple of $\vol_V$, meaning that $c\ge0$.
    
    The same argument shows that if $\Omega$ is in the interior of the cone, so is $\rest{\vartheta}{\mathfrak h}$.
\end{proof}

A particular instance of this argument can be found in the proofs of \cite[Prop.~2.5, 3.3]{fm}.

\smallskip
Using the results obtained in \Cref{htoV}, we can in fact prove that  the converse of \Cref{nonso} is also true, when $\eta=0$.
More precisely, if $\eta=0$, the converse of the first part of the statement follows combining \Cref{coneshp,coneswedgeh}, whereas the second part is the following.

\begin{theorem}\label{cono+conoincono}
    For all $v_0\in V$, $\alpha\in\Lambda^{1,0}\setminus\rv\pt{\lambdahh 10}$, and for  $\cone=\spos{},\hpos{},\wpos{}$,  we have
    \begin{equation*}
        \rv\pt{\sconeh^{p}}+i\,\alpha\wedge\bar\alpha\wedge\rv\pt{\sconeh^{p-1}}\subseteq\scone^p,
    \end{equation*}
    $p=1,\dots,n-1$, namely, if $\Xi\in\rv\pt{\sconeh^p}, \Psi\in\rv\pt{\sconeh^{p-1}}$, then $\Xi+i\,\alpha\wedge\bar\alpha\wedge\Psi\in\scone^p$.
\end{theorem}
\begin{proof}
    Let $\Xi\in\rv\pt{\sconeh^p}, \Psi\in\rv\pt{\sconeh^{p-1}}$.
    We know that $\Theta\coloneqq\Xi+i\,\alpha\wedge\bar\alpha\wedge\Psi\in\cone^p$, by \Cref{coneswedgeh}.
    To prove that $\Theta$ is in fact in $\scone^p$, we will prove that $\Theta\wedge\phi>0$, for all $0\neq\phi\in\pt{\cone^{p}}\check{}$.
    Writing such a $\phi$ as in \eqref{restOmega}, 
    \begin{equation*}
        \phi=\rv\pt{\rest{\phi}{\mathfrak h}}+\alpha\wedge\rv\eta+\bar\alpha\wedge\bar\rv\eta+i\,\alpha\wedge\bar\alpha\wedge\xi,
    \end{equation*}
    with $\rv\eta\in\rv\lambdahh{n-p-1}{n-p}, \xi\in\rv\lambdahh{n-p-1}{n-p-1}   $, by \Cref{rmkDualCone} and \Cref{nonso} we know that 
    \begin{equation*}
        \rest{\phi}{\mathfrak h}\in\pt{\coneh^{p-1}}\check{},\quad
        \rest{\xi}{\mathfrak h}\in\pt{\coneh^{p}}\check{}.
    \end{equation*}
    This is enough to conclude that
    \begin{equation*}
        \Theta\wedge\phi=i\,\alpha\wedge\bar\alpha\wedge\pt{\Xi\wedge\xi+\Psi\wedge\rest{\phi}{\mathfrak h}}>0.
    \end{equation*}
\end{proof}

We conclude this section with a remark on the \textit{diagonalization} of Hermitian positive forms obtained via this construction.
We denote the dimensions of the spaces of $(p,0)$-forms as
\begin{equation*}
    N^p\coloneqq\dim\Lambda^{p,0}=\binom{n}{p},\quad N^p_\mathfrak h\coloneqq\dim\lambdahh p0=\binom{n-1}{p}.
\end{equation*}

\begin{remark}\label{p0formsh}
    Let $p=1,\dots,n$.
    Then, for any fixed $\alpha\in\Lambda^{1,0}\setminus\rv\lambdahh 10$, we have 
    \begin{equation*}
        \Lambda^{p,0}=\rv\lambdahh p0\oplus\alpha\wedge\rv\lambdahh{p-1}0.
    \end{equation*}
    More precisely, if $\pg{\varphi_1,\dots,\varphi_{N^p_\mathfrak h}}$ is a basis of $\lambdahh p0$, and $\pg{\phi_1,\dots,\phi_{N^{p-1}_\mathfrak h}}$ is a basis of $\lambdahh {p-1}0$, then 
\begin{equation*}
    \pg{\rv\varphi_1,\dots,\rv\varphi_{N^p_\mathfrak h},\alpha\wedge\rv\phi_1,\dots,\alpha\wedge\rv\phi_{N^{p-1}_\mathfrak h}}
\end{equation*}
is a basis of $\Lambda^{p,0}$, because it is clearly a linearly independent set, and 
    \begin{equation*}
         N^p_\mathfrak h+N^{p-1}_\mathfrak h=\binom{n-1}{p}+\binom{n-1}{p-1}=\binom{n}{p}=\dim\Lambda^{p,0}.
    \end{equation*}
\end{remark}

\section{\texorpdfstring{$(2,2)$}{(2,2)}-forms on \texorpdfstring{$\C^4$}{C4}}\label{wp224}

This is the first non-trivial case, as it is the lowest dimension case where one can find examples of weakly positive forms that are not Hermitian positive and (by duality) Hermitian positive forms that are not strongly positive.
In this case, for a fixed $4$-dimensional vector space $V$ over $\C$, the image through the Pl\"ucker embedding of the Grassmannian $\mathbb{G}_2(V)$ in $\mathbb P(\Lambda^2 V) \cong \mathbb P^6$ is identified by one equation.
Without loss of generality, we regard $V$ as $\C^4$.
Following the notation of the previous sections, we denote $\Lambda^{p,q}=\Lambda^{p,q}\pt{\C^4}^*$, for $p,q=1,\dots,4$.
Fix a basis $\omega^1,\dots,\omega^4$ of $\Lambda^{1,0}$, and consider the volume form $\vol=\omega^{1234\bar1\bar2\bar3\bar4}$ on $\C^4$.
Then, a basis of $\Lambda^{2,0}$ is $\pg{\varphi^1,\dots,\varphi^6}$, defined by
\begin{equation}\label{base20}\begin{aligned}
   & \varphi^1=\omega^{12},\quad\varphi^2=\omega^{13},\quad\varphi^3=\omega^{14},\\
   & \varphi^4=\omega^{23},\quad\varphi^5=-\omega^{24},\quad\varphi^6=\omega^{34}.
\end{aligned}
\end{equation}
This particular basis satisfies 
\begin{equation*}
\varphi^j\wedge\varphi^k\wedge\varphi^{\bar j}\wedge\varphi^{\bar k}=
\begin{cases}
    \text{Vol}& (j,k)=(1,6),(2,5),(3,4),\\
    0   &   \text{otherwise}.
\end{cases}
\end{equation*}

\subsection{Criteria for weak positivity}\label{secWpos}

Once we fix the basis \eqref{base20}, we find the following criterion to determine when a $(2,0)$-form on $V$ is decomposable, using \Cref{decomPlucker} (cf. \cite{BlockiPlis}).

\begin{proposition}
    The $(2,0)$-form $\psi=\sum_{j=1}^6c_j\varphi^j$ is decomposable if and only if 
    $$
    c_1c_6+c_2c_5+c_3c_4=0.
    $$
\end{proposition}

Therefore, in light of \Cref{posQ}, we obtain a criterion for weak positivity (cf. \cite[Thm.~3]{BlockiPlis}) and strict weak positivity.

\begin{proposition}\label{wposA}
 A real $(2,2)$-form $\Theta$ is weakly positive, respectively strictly weakly positive, if and only if $Q_\Theta$ is positive semi-definite, respectively positive definite, on
    \begin{equation*}
        \mathcal{P}l=\pg{z_1z_6+z_2z_5+z_3z_4=0}\subset\C^6\cong \Lambda^{2,0}.
    \end{equation*}
\end{proposition}

Here, and from now on, the identification $\Lambda^{2,0}\cong\C^6$ is intended with respect to the basis $\pg{\varphi^1,\dots,\varphi^6}$ defined above.

\begin{notation}
    Any real $(2,2)$-form $\Theta$ is identified by the Hermitian matrix $A_\Theta$ associated to $Q_\Theta$ in the basis $\pg{\varphi^6,\dots,\varphi^1}$.
    Note that, if $A_\Theta=(a_{j,k})_{j,k=1}^6$, then  
    \begin{equation}\label{defTheta}
    \Theta=\sum_{1\leq j, k\leq6}a_{j,k}\,\varphi^j\wedge\varphi^{\bar k},
    \end{equation}
    where $\pg{\varphi^j\wedge\varphi^{\bar k}}_{j,k=1}^6$ is the corresponding basis of $\Lambda^{2,2}$.
    % for some Hermitian matrix $(a_{j,k})_{j,k=1}^6$.
    % From now on, we will denote this matrix with $A_\Theta$.
\end{notation}

It follows from the characterizations above that if $\Theta$ is weakly positive, respectively strictly weakly positive, then for any $j<k$ with $j+k\neq7$, the submatrix 
\begin{equation*}
A_{j,k}=
\begin{pmatrix}
    a_{j,j} & a_{j,k}   \\
    \bar a_{j,k} & a_{k,k}       
\end{pmatrix}\end{equation*}
 of $A_\Theta$ is positive semi-definite, respectively positive definite.
To check this, let $v\in\C^6$ be the coordinates vector of $\psi=z_j\varphi^{7-j}+z_k\varphi^{7-k}$.
Since $j+k\neq7$, we have $v\in\mathcal{P}l$.
Then,
    \begin{equation*}\label{mandorlina}\begin{aligned}        
        0 &\le  \frac{\Theta\wedge\psi\wedge\bar\psi}{\vol} 
        = Q_\Theta\pt{\psi,\psi} 
        =  v\,A_\Theta\,\overline{v}^t\\
        &= (z_j,z_k)\begin{pmatrix}
    a_{j,j} & a_{j,k}   \\
    \bar a_{j,k} & a_{k,k}       
\end{pmatrix}
\begin{pmatrix}
    \bar z_j   \\ \bar z_k    
\end{pmatrix}.
        % {\overline{6} }\wedge\varphi^{5\overline{5}}\wedge\varphi^{7-j\overline{5}}\wedge\varphi^{5\overline{7-j}}
    \end{aligned}    \end{equation*}
This shows how the elements outside the secondary diagonal of $A_\Theta$ behave similarly to the case of Hermitian positivity.
In particular, this gives bounds on such elements.

\begin{proposition}\label{+diag}
Let $\Theta\in\wpos 2$.
The diagonal elements of $A_\Theta$ are non-negative.
Furthermore, for all $j\neq k$ and $j+k\neq 7$, we have
\begin{equation}\label{nondiag}
    \abs{a_{j,k}}\le\frac{a_{j,j}+a_{k,k}}{2}\le\frac{\tr A_\Theta}{2}.
\end{equation}
Analogously, if $\Theta\in\swpos 2$, the strict inequalities hold.
\end{proposition}

\begin{proof}
    The first part of the statement is trivial.
    To prove \eqref{nondiag}, notice that if $A_{j,k}$ is positive semi-definite, then
    \begin{equation*}
        a_{j,j}+a_{k,k} -2\abs{a_{j,k}}
        \ge a_{j,j}+a_{k,k} -2\sqrt{a_{j,j}a_{k,k}}
        =\pt{\sqrt{a_{j,j}}-\sqrt{a_{k,k}}}^2\ge0.
    \end{equation*}
    Moreover, if $\Theta$ is strictly weakly positive, the first inequality is strict.
\end{proof}

This suggests that weakly positive forms that are not Hermitian positive can be identified via conditions on the secondary diagonal.

\begin{theorem}\label{3antidiag}
    Let $\Theta$ be a real $(2,2)$-form on $\C^4$, and $A_\Theta=(a_{j,k})$ its associated Hermitian matrix with respect to the basis $\pg{\varphi^1,\dots,\varphi^6}$.
    If $\Theta$ is weakly positive, then for all $(j,k)\neq(l,m)$ such that $j+k=l+m=7$,
   \begin{equation}\label{wposboh}
        \abs{a_{j,k}\pm a_{l,m}}\le \frac{a_{j,j}+a_{k,k}+a_{l,l}+a_{m,m}}2. 
    \end{equation}
   % \begin{equation}\label{wposboh}\begin{aligned}
   %      \abs{a_{1,6}-a_{2,5}}\le t_{1,2,5,6},\quad\abs{a_{1,6}-a_{3,4}}\le T,\quad\abs{a_{2,5}-a_{3,4}}\le T,\\
   %      \abs{a_{1,6}+a_{2,5}}\le T,\quad\abs{a_{1,6}+a_{3,4}}\le T,\quad\abs{a_{2,5}+a_{3,4}}\le T,
   %  \end{aligned}    \end{equation}2$.    
    % where $T=\frac{\tr A}2$.
    In particular, $a_{1,6},a_{2,5},a_{3,4}$ have norm less than or equal to $\displaystyle\frac{\tr A_\Theta}2$.
\end{theorem}

\begin{remark}\label{rmktr/2}
    We have just seen that each non-diagonal entry $a_{j,k}$ of the matrix $A_\Theta$ associated to a weakly positive $(2,2)$-form $\Theta$ satisfies
    \begin{equation}\label{tr/2}
    \abs{a_{j,k}}\le\frac{\tr A_\Theta}{2}.
    \end{equation}
    However, as we will see in the next section, the converse does not hold.
\end{remark}

\begin{proof}[Proof of \Cref{3antidiag}]
We use the criterion of \Cref{wposA} for weak positivity in terms of the bilinear form $Q$ on $\Lambda^{2,0}\simeq\C^6$.
In this case, we have, for $(z_1,\dots,z_6)\in\C^6$,
\begin{equation*}
    Q(z_1,\dots,z_6)=\sum_{j=1}^6\pt{a_{j,j}\abs{z_j}^2+2\sum_{k=j+1}^6\real\pt{a_{j,k}\bar z_jz_k}}.
\end{equation*}
If $\Theta$ is weakly positive, then $Q(z_1,\dots,z_6)\ge0$, for all $(z_1,\dots,z_6)\in\mathcal{P}l$, namely such that $z_1z_6+z_2z_5+z_3z_4=0$.

Set $t=a_{1,1}+a_{2,2}+a_{5,5}+a_{6,6}$.
We will now prove that $\abs{w}\le \frac{t}{2}$, with $w=a_{2,5}-a_{1,6}$.
Let us consider
\begin{equation*}
    (-1,1,0,0,z,z),(1,1,0,0,z,-z)\in\mathcal{P}l,
\end{equation*}
for all $z\in\C$.
Then, $Q_0(z)\coloneqq Q(1,1,0,0,z,-z)+Q(-1,1,0,0,z,z)$ should be non-negative, for $z=x+iy\in\C$, for all $x,y\in\R$.
We have
\begin{equation*}\begin{aligned}    
    Q_0(z)&=2\pt{\abs{z}^2(a_{5,5}+a_{6,6})+2\real\pt{z\,{w}}+a_{1,1}+a_{2,2}}\\
    &=2\pt{\pt{x^2+y^2}(a_{5,5}+a_{6,6})+2x\real{{w}}-2y\imm{{w}}+a_{1,1}+a_{2,2}} \ge0.
\end{aligned}\end{equation*}
We see this as a second degree inequality in the real variable $x$, so we know that it is satisfied for all $x\in\R$ if and only if its discriminant $\Delta_{x}$ is always non-positive, with
\begin{equation*}
    \frac{\Delta_{x}}4=\pt{\real{w}}^2-(a_{5,5}+a_{6,6})\pt{a_{1,1}+a_{2,2}-2y\imm w+y^2(a_{5,5}+a_{6,6})}\le 0.
\end{equation*}
% for all $y\in\R$.
This last line is now a second degree inequality in $y$, equivalent to 
\begin{equation*}
    \frac{\Delta_{y}}4=(a_{5,5}+a_{6,6})^2\pt{\pt{\imm w}^2+{\pt{\real{w}}^2-(a_{5,5}+a_{6,6})(a_{1,1}+a_{2,2})}}
    \le0,
\end{equation*}
namely $\abs{w}^2\le(a_{5,5}+a_{6,6})(a_{1,1}+a_{2,2})$.
Now, consider $t_0=a_{1,1}+a_{2,2}$, so that $a_{5,5}+a_{6,6}=t-t_0$.
We have 
$$
\abs{w}^2\le (t-t_0)t_0= -\pt{t_0-\frac{t}{2}}^2+\frac{t^2}{4}\le \frac{t^2}{4},
$$
so that $\abs{w}\le\frac{t}{2}$.
In other words,
$$
\abs{a_{1,6}-a_{2,5}}\le\frac{a_{1,1}+a_{2,2}+a_{5,5}+a_{6,6}}{2}.
$$
Similarly, $\tilde Q_0(z) = Q(i,1,0,0,z,i\,z)+Q(-i,1,0,0,z,-i\,z)\ge0$, for all $z\in\C$, with
\begin{equation*}
     \tilde Q_0(z)=2\pt{\abs{z}^2(a_{5,5}+a_{6,6})+2\real\pt{z\,\pt{a_{2,5}+a_{1,6}}}+a_{1,1}+a_{2,2}},
\end{equation*}
ultimately yielding
$$
\abs{a_{1,6}+a_{2,5}}\le\frac{a_{1,1}+a_{2,2}+a_{5,5}+a_{6,6}}{2}.
$$

Considering $Q(\pm1,0,1,z,0,\mp z)$ and $Q(\pm i,0,1,z,0,\pm i\,z)$ we get
\begin{equation*}
    \abs{a_{1,6} \pm a_{3,4}}\le\frac{a_{1,1}+a_{3,3}+a_{4,4}+a_{6,6}}{2} .
\end{equation*}
Similarly, considering $Q(0,\pm 1,1,z,\mp z,0)$ and $Q(0,\pm i,1,z,\pm i\,z,0)$ we get
\begin{equation*}
    \abs{a_{2,5} \pm a_{3,4}}\le\frac{a_{2,2}+a_{3,3}+a_{4,4}+a_{5,5}}{2}
\end{equation*}
thus the remaining cases are proved.

The last part of the statement is a trivial consequence of the inequalities we just proved, as \eqref{wposboh} yields
% \begin{equation*}\begin{aligned}    
% -\frac{a_{1,1}+a_{2,2}+a_{5,5}+a_{6,6}}{2}+a_{2,5}\le {a_{1,6}}\le\frac{a_{1,1}+a_{2,2}+a_{5,5}+a_{6,6}}{2}+a_{2,5},\\
% -\frac{a_{1,1}+a_{2,2}+a_{5,5}+a_{6,6}}{2}-a_{2,5}\le {a_{1,6}}\le\frac{a_{1,1}+a_{2,2}+a_{5,5}+a_{6,6}}{2}-a_{2,5}.
% \end{aligned}\end{equation*}
% frase
\begin{equation*}\begin{aligned}    
\abs{a_{j,k}}\le\frac{a_{j,j}+a_{k,k}+a_{l,l}+a_{m,m}}2+a_{l,m},\\
\abs{a_{j,k}}\le\frac{a_{j,j}+a_{k,k}+a_{l,l}+a_{m,m}}2-a_{l,m},
\end{aligned}\end{equation*}
and from this it follows the inequality~\eqref{gigettino}, namely 
\begin{equation*}
\abs{a_{j,k}}\le\frac{a_{j,j}+a_{k,k}+a_{l,l}+a_{m,m}}2.
\end{equation*}
\end{proof}

\begin{remark}\label{rmkprefe}
    In the last line of the proof, we actually obtained a bound on the norm of the $a_{j,k}$ with $j+k=7$, which is better then half the trace.
    Note that if $\Theta$ is also Hermitian positive, the bound \eqref{gigettino} can be easily improved to \eqref{nondiag}.
    However, as we will see later, this is not the case for $\Theta\in\wpos 2\setminus\hpos 2$.
\end{remark}

\begin{corollary}\label{tr>0}
    For every non-zero weakly positive $(2,2)$-form $\Theta$, the diagonal elements of $A_\Theta$ are not all zero.
In particular,
\begin{equation*}
    \tr A_\Theta>0.
\end{equation*}
\end{corollary}

\begin{proof}
    This is a straightforward consequence of \eqref{tr/2}, forcing $A_\Theta$ to be the zero matrix if it has null trace. 
\end{proof}

\Cref{+diag} can be generalized to $3\times3$ submatrices of $A_\Theta$, not intersecting the secondary diagonal, namely 
\begin{equation*}
    A_{j,k,l}=\begin{pmatrix}
        a_{j,j} &   a_{j,k} &   a_{j,l} \\
        a_{k,j} &   a_{k,k} &   a_{k,l} \\
        a_{l,j} &   a_{l,k} &   a_{l,l} \\
    \end{pmatrix},
\end{equation*}
for all $j<k<l$ such that $(j,k),(k,l),(j,l)\neq(1,6),(2,5),(3,4)$.
Using \Cref{Vtoh}, we can associate such matrices to forms of bidegree $(1,1)$ and $(2,2)$ on a complex subspace of $\C^4$ of  dimension $3$.
More precisely, we can prove the following.

\begin{proposition}\label{3x3subm}
    Let $\Theta$ be a  weakly positive $(2,2)$-form on $\C^4$.
    Then, for all $j<k<l$ such that $(j,k),(k,l),(j,l)\neq(1,6),(2,5),(3,4)$, the matrix $A_{j,k,l}$ is positive semi-definite.
    Furthermore, if $\Theta$ is strictly weakly positive, then $A_{j,k,l}$ has to be positive definite.
\end{proposition}

\begin{proof}
    Let $v_j=\pt{\omega^j}^\sharp\in\C^4$, for $j=1,\dots,4$ .
    We can define $\mathfrak h_1=\text{span}\pg{v_2,v_3,v_4},$
    and $\alpha\coloneqq\omega^1$ is an element in $\Lambda^{1,0}\setminus\rho_{v_1}\lamhh110$ with $\alpha(v_1)=1$, so we can write $\Theta$ as in \eqref{restOmega}:
\begin{equation}\label{restTheta}
    \Theta=\Xi_1+\alpha\wedge\rho_{v_1}\eta_1+\bar\alpha\wedge\rho_{v_1}\bar\eta_1+i\alpha\wedge\bar\alpha\wedge\bar\vartheta_1,
\end{equation}
with $\Xi_1\in\rho_{v_1}\lamhh122,\eta_1\in\lamhh112,\vartheta_1\in\rho_{v_1}\lamhh111$
Comparing \eqref{restTheta} with \eqref{defTheta}, we can see that $A_{4,5,6}$ is the matrix associated to $\Xi_1$, in the basis $\pg{\varphi^4,\varphi^5,\varphi^6}$ of $\lamhh120
$.
Furthermore, 
\begin{equation*}
    \vartheta_1=i\,\sum_{j,k=1}^3\,a_{j,k}\,\omega^{j+1}\wedge\omega^{\overline{k+1}}\in\lamhh111
\end{equation*} 
has associated matrix $A_{1,2,3}$.
The analogous construction, with  $\alpha\coloneqq\omega^j$, and ${\mathfrak h}_j=\text{span}\pg{v_1,\dots,\hat v_j,\dots,v_4}$, for $j=2,3,4$, gives $\Xi_j\in\lamhh j22,\vartheta_j\in\lamhh j11$, with associated matrices
\begin{equation*}\begin{aligned}
    &A_{\Xi_2}=A_{2,3,6},\quad  A_{\Xi_3}=A_{1,3,5},\quad  A_{\Xi_4}=A_{1,2,4},\\
    &A_{\vartheta_2}=A_{1,4,5},\quad A_{\vartheta_3}=A_{2,4,6},\quad A_{\vartheta_4}=A_{3,5,6}.
\end{aligned}
\end{equation*}
By \Cref{nonso}, if $\Theta$ is weakly positive, then so are $\Xi_j$ and $\vartheta_j$, for all $j=1,\dots,4$, but for all such forms, this is equivalent to be Hermitian positive, and similarly for strict weak positivity, so this completes the proof.
\end{proof}

Note that this can also be proved considering the points in $\C^6$ with $z_j,z_k,z_l\neq0$, and the remaining $z_r$ all null, which are in $\mathcal{P}l$ for the values of $j,k,l$ as in the statement.

\subsection{Applications for strong positivity}\label{secSpos}
We will now use what we proved in \Cref{secWpos} to study strong positivity on a family of $(2,2)$-forms, comparing it with the other positivity notions.
Let
    $$
    \Omega_a=\displaystyle\sum_{1\leq l\leq6}\varphi^l\wedge\varphi^{\bar l}+a\,\varphi^j\wedge\varphi^{\bar k}+\bar a\,\varphi^k\wedge\varphi^{\bar l},
    $$
for $a\in\C$, with $j<k$ fixed.
Note that we can write 
\begin{equation}\label{questa}
    \Omega_a=\sum_{l\neq j,k}\varphi^l\wedge\varphi^{\bar l}+\pt{\varphi^j+\bar a\,\varphi^k}\wedge{\pt{\varphi^{\bar j}+ a\,\varphi^{\bar k}}}+\pt{1-\abs a^2}\varphi^k\wedge\varphi^{\bar k}.
\end{equation}
If $\abs a\le1$, this is a convex sum of Hermitian positive forms, thus Hermitian positive.
Furthermore, it is clear that $\Omega_0$ is strictly strongly positive.

\begin{lemma}\label{HSb}
    For the family $\pg{\Omega_a}_{a\in\C}$, the following are equivalent.
    \begin{enumerate}[label=\roman*.]
        \item\label{hsb1} $\Omega_a$ is Hermitian positive
        \item\label{hsb2} $\Omega_a$ is strongly positive
        \item\label{hsb3}  $\abs a\leq1$.
    \end{enumerate} 
    The analogous holds for strict positivity, with $\abs a<1$.
\end{lemma}

\begin{proof}
    Note that \eqref{hsb1} $\Rightarrow$ \eqref{hsb3} follows by Sylvester's criterion, as the principal minors of $A_{\Omega_a}$ are all either $1$ or $1-\abs a^2$.
    
    We now prove that \eqref{hsb3} $\Rightarrow$ \eqref{hsb2}.    
    We show that if $\abs a\le1$, then the wedge product of $\Omega_a$ with any weakly positive $(2,2)$-form is always non-negative.
    Let $\Theta$ be such a form, and consider $A_\Theta=(a_{l,m})$.
    Then,
    \begin{equation}\label{sivaapranzo}
        \frac{\Omega_a\wedge\Theta}{\vol}=\tr A_\Theta+2\real\pt{a\,\bar a_{j,k}}
        \ge\tr A_\Theta-2\abs a\abs{a_{j,k}}
        \ge \tr A_\Theta -2\abs{a_{j,k}},
    \end{equation}
    where the last inequality holds because $\abs a\le1$.
    The last quantity is non-negative by \Cref{rmktr/2}.

    The same arguments hold for the strict positivity notions.
    The only non-trivial part is to prove that $\Omega_a\wedge\Theta>0$ if $\abs{a}<1$, for every choice of $\Theta\neq0$.
    We have two cases.
    If $a_{j,k}\neq0$, the last inequality in \eqref{sivaapranzo} is strict.
    Otherwise, the thesis follows because $\tr A_\Theta>0$, by \Cref{tr>0}. 
\end{proof}

We will now consider two cases, depending on the choice of $j$ and $k$ in the definition of $\Omega_a$.
More precisely, we will show that the main difference is given by $a$ being in the secondary diagonal or not, highlighting what we anticipated in \Cref{rmkprefe}.
This will ultimately allow us to build examples of weakly positive forms which are not Hermitian positive.

\begin{proposition}\label{wposb}
If $j+k\neq7$, then all the conditions in \Cref{HSb} are equivalent to $\Omega_a$ being weakly positive, respectively strictly weakly positive.
\end{proposition}

\begin{proof}
    It is enough to show that if $\Omega_a$ is weakly positive, respectively  strictly weakly positive, then $\abs{a}\le1$, respectively $\abs a<1$.
    This is a particular case of \Cref{+diag}.
\end{proof}

On the other hand, if $j+k=7$, $\Omega_a$ is the family $\alpha_a$ introduced in \cite[Prop.~4]{BlockiPlis}, where the weak positivity was studied.
Our aim is to extend this result to the remaining types of positivity.
Using \Cref{HSb}, we obtain this in the following.

\begin{theorem}\label{alphaa}
If $j+k = 7$, then
    \begin{enumerate}[label=\arabic*.]
    \item\label{bp} $\Omega_a$ is weakly positive if and only if $\abs a\leq2$.
    \item\label{alphaatransv} $\Omega_a$ is strictly weakly positive if and only if $\abs a<2$.
    \item\label{gigi} The following are equivalent.
        \begin{enumerate}[label=\roman*.]
        \item\label{gigi1} $\Omega_a$ is Hermitian positive;
        \item\label{gigi2} $\Omega_a$ is strongly positive;
        \item\label{gigi3} $\abs a\leq1$.
        \end{enumerate}
    \item\label{gigis} The following are equivalent.
        \begin{enumerate}[label=\roman*.]
        \item\label{gigis1} $\Omega_a$ is strictly Hermitian positive;
        \item\label{gigis2} $\Omega_a$ is strictly strongly positive;
        \item\label{gigis3} $\abs a<1$.
        \end{enumerate}
    \end{enumerate}
\end{theorem}

\begin{proof}
    Without loss of generality, up to reordering the basis, we can assume $j=1$, $k=6$.
    \Cref{bp} is then \cite[Prop.~4]{BlockiPlis}.

    From now on, let 
    \begin{equation*}
        A\coloneqq A_{\Omega_a}=\begin{pmatrix}
            1&0&0&0&0&a\\
            0&1&0&0&0&0\\
            0&0&1&0&0&0\\
            0&0&0&1&0&0\\
            0&0&0&0&1&0\\
            \bar a&0&0&0&0&1\\
        \end{pmatrix}.
    \end{equation*}    
    Retracing the steps of the proof of \cite[Prop.~4]{BlockiPlis}, we find that for $z=(z_1,\dots,z_6)\in\mathcal{P}l$, we have
    \begin{equation}\label{bplis}\begin{aligned}
        \bar zAz^t=\abs{z}^2+2\text{Re}\pt{a\bar z_1z_6}
        &\geq 2\abs{z_1z_6}+2\abs{z_2z_5}+2\abs{z_3z_4}+2\text{Re}\pt{a\bar z_1z_6}  \\
        &\geq 2\abs{z_1z_6}+2\abs{z_2z_5+z_3z_4}+2\text{Re}\pt{a\bar z_1z_6}  \\
        &= 2\pt{2\abs{z_1z_6}+\text{Re}\pt{a\bar z_1z_6}}  \\
        &\geq2\pt{2\abs{z_1z_6}-\abs{a\bar z_1z_6}}  \\
        &=2\abs{z_1z_6}\pt{2-\abs{a}}  \geq0.
    \end{aligned}
    \end{equation}
    Furthermore, for $w\in\mathcal{P}l\cap\pg{\bar z_1z_6=-\bar a,\abs{z_1}=\abs{z_6}}$, this reduces to
    \begin{equation*}
        \bar w A w^t=2\abs{a}\pt{2-\abs a}.
    \end{equation*}
    A straightforward consequence is that if $\Omega_a$ is strictly weakly positive, then $\abs a<2$.
    
    For the viceversa, we only have to prove that the restriction of $Q_{\Omega_a}$ to $\mathcal{P}l$ is non-degenerate, when $\abs a<2$, namely that if $\bar zAz^t=0$ for some $z\in\mathcal{P}l$, then $z=0$.
    To see this, we note that the inequalities in \eqref{bplis} are in fact equalities, respectively, if and only if 
    \begin{equation*}\begin{aligned}
        &\abs{z_1}=\abs{z_6},\quad\abs{z_2}=\abs{z_5},\quad\abs{z_3}=\abs{z_4},\\
        &{z_2z_5}=\lambda{z_3z_4},\quad\lambda\in\C,\\
        &a{\bar z_1z_6}\in\R_{\leq0},\\
        &\abs{z_1z_6}=0.
    \end{aligned}
    \end{equation*}
    Now the first line,  together with the last one, gives $z_1=z_6=0$, so that
    $$
    0=\bar zAz^t=\abs{z}^2,
    $$
    and $z=0$, as wanted.

    \Cref{gigi,gigis} are the content of \Cref{HSb}.
\end{proof}

A part of this proof was used in \cite{LT}, to prove that a given $(3,3)$-form on a compact holomorphically parallelizable manifold is $3$-K\"ahler.

\smallskip
We summarize the statement of \Cref{alphaa} in the following table. 
\medskip

\renewcommand{\arraystretch}{1.7} 
\begin{center}
\begin{tabular}{|c|c|c|c|c|c|c|}
\hline
$\Omega_a\in$     & $\wpos 2$ & $\swpos 2$ & $\hpos 2$ & $\shpos 2$ & $\spos 2$ & $\sspos 2$ \\ \hline
$\abs{a}<1$   & \tic      & \tic       & \tic      & \tic       & \tic      & \tic       \\ \hline
$\abs{a}=1$   & \tic      & \tic       & \tic      & \no        & \tic      & \no        \\ \hline
$1<\abs{a}<2$ & \tic      & \tic       & \no       & \no        & \no       & \no        \\ \hline
$\abs{a}=2$   & \tic      & \no        & \no       & \no        & \no       & \no        \\ \hline
$\abs{a}>2$   & \no       & \no        & \no       & \no        & \no       & \no        \\ \hline
\end{tabular}
\end{center}

\smallskip
\begin{remark}
    Let us consider the case  $0<\abs a \le 1$.
    Observe that $\varphi^j+\bar a\varphi^k$, with $j+k=7$, are non-decomposable $(2,0)$-forms, and thus
    \[
    \Psi_a\coloneqq\pt{\varphi^j+\bar a\varphi^k}\wedge\pt{\varphi^{\bar j}+a\varphi^{\bar k}}
    \]
    are examples of $\rk1$ Hermitian positive forms that are not strongly positive, cf. \cite[Prop. 1.5]{hk}.
    Also, we noted in \eqref{questa} that we can write $\Omega_a$ as the sum of $\Psi_a$ and a strongly positive part.
    A priori, this could suggest $\Omega_a$ to be candidates for Hermitian positive forms which are not strongly positive.
    However, since we proved that $\Omega_a$ are strongly positive, this could be a starting point to get more information on the geometry of the cones $\spos{2},\hpos 2$ and on the reciprocal positions of the strongly positive summand with respect to $\Psi_a$.
    % The heuristic interpretation should be that the strongly positive summand has norm big enough to bring $\Psi_a$ in the cone of strongly positive forms.
\end{remark}

We conclude noting that \Cref{bp} in \Cref{alphaa} provides a counterexample of the viceversa of \Cref{rmktr/2}.
Indeed, any form $\Omega_a$ with $2<\abs a\le 3$ satisfies 
\begin{equation*}
    \abs{a_{j,k}}\le\tr A/2,\quad \text{for all }j\neq k,
\end{equation*}
but is not weakly positive.
Moreover, the weakly positive forms $\Omega_a$ with $\abs a = 2$ proves that the upper bound~\eqref{gigettino} is sharp, as anticipated in \Cref{rmkprefe}.
\bigskip

{\bf Acknowledgments.}
The second author would like to thank Anna Fino for the constant support during the PhD supervision.
The authors would like to warmly thank Simone Diverio for useful comments.

%\nocite{*} % mette anche le referenze che non citi
\bibliographystyle{plain}
\bibliography{positivityFM}
% \printbibliography
\end{document}